\documentclass[a4paper,10pt]{article}

 \usepackage{amsmath,amsfonts,amssymb,amsthm,mathrsfs}
\usepackage[hidelinks]{hyperref}
 \usepackage{bm}
 \usepackage{color}
 \usepackage{xcolor}
 \usepackage{esint}
\usepackage{graphicx}
\usepackage{graphics}
\usepackage{geometry}
\usepackage[all]{xy}

\usepackage{tikz}

\newtheorem{thm}{Theorem}

\newtheorem{lem}[thm]{Lemma}

\newtheorem{prop}[thm]{Proposition}

\theoremstyle{definition}

\def \R {\mathbb R}
\def \D {\mathcal D}

\def \M {\mathcal M}

\def \K {\mathcal K}

\def \dvol {\text{dVol}}

\begin{document}
\title{A short proof of the surjectivity of the period map on K3 manifolds}
\author{Hongyi Liu}
\date{}
\maketitle
\begin{abstract}
    In this note, we give a simple proof of the Todorov's surjectivity result on the period map of K3 surfaces in a differential geometric setting. Our proof makes use of collasping geometry of hyperk\"ahler 4-manifolds developped by Sun-Zhang in \cite{sun2021collapsing}, and does not rely on the solution to the Calabi conjecture.
    
\end{abstract}

On an oriented smooth 4-manifold, a hyperk\"ahler metric $g$ is a Riemannian metric with holonomy contained in $SU(2)$. This is equivalent to saying that the bundle of self-dual forms is flat and trivial, which implies that there is a triple $\bm\omega=(\omega_1, \omega_2, \omega_3)$ of closed 2-forms  satisfying  $\omega_\alpha\wedge\omega_\beta=2\delta_{\alpha\beta} \dvol_g$. Such a triple is often referred to as a  hyperk\"ahler triple.

Let $\mathcal K$ be the K3 manifold, which is by definition the unique oriented smooth 4-manifold underlying a complex K3 surface. It is simply connected and the intersection form on  $\Lambda\equiv H^2(\K; \mathbb Z)$ has signature $(3, 19)$. 
Denote by $\mathcal N$ the set of all hyperk\"ahler metrics on $\mathcal K$ with diameter 1. 
  Given $g\in \mathcal N$,  the space $\mathbb H^+_g$ of self-dual harmonic 2-forms with respect to $g$ is a 3-dimensional subspace  in $\Lambda_{\mathbb R}\equiv H^2(\mathcal K; \mathbb R)$ which is positive definite with respect to the intersection form. Indeed, any choice of a hyperk\"ahler triple gives rise to a basis of $\mathbb H_g^+$, and they are up to a constant $O(3)$ rotation.  

Define the positive Grassmannian $Gr^{+}$ to be the space of all 3-dimensional positive definite subspaces of $\Lambda_{\mathbb R}$. It is an open subset in the standard Grassmannian $Gr(3, \Lambda_\R)$. 
 We define the period map
$$\mathcal P: \mathcal N\rightarrow  Gr^{+}; g\mapsto \mathbb H_g^{+}.$$
The diffeomorphism group $\textbf{Diff}(\K)$ acts on $\mathcal N$ by $\varphi.g=\varphi^*g$, which induces a homomorphism $\Phi: \textbf{Diff}(\K)\rightarrow \text{Aut}(\Lambda)$, where $\Gamma:=\text{Aut}(\Lambda)$ is the automorphism group of the lattice $\Lambda$ preserving the intersection form. 
There is a natural action of  $\Gamma$  on $Gr^{+}$, hence $\mathcal{P}$ induces a map
\begin{equation}\label{period}
     \underline{\mathcal P}: \mathcal M=\mathcal N/\textbf{Diff}(\mathcal K)\rightarrow \mathcal D\equiv Gr^{+}/\Gamma.
\end{equation}
The left-hand side is the set of isometry classes of hyperk\"ahler metrics on $\mathcal K$. It is endowed with a natural Cheeger-Gromov topology.
A sequence $[g_j]$ converges to $[g_\infty]$ if there are $\varphi_j\in \textbf{Diff}(\mathcal K)$ such that $\varphi_j^*g_j$ converges smoothly to $g_\infty$.
Since hyperk\"ahler metrics are Ricci-flat, it follows from the Cheeger-Colding theory that this topology coincides with the Gromov-Hausdorff topology. We also endow the period domain $\mathcal D$ with the quotient topology. One can check that $\underline{\mathcal P}$ is continuous, and the $\Gamma$ action on $Gr^+$ is properly discontinous.
 
 For any homology class $\delta\in H_2(\mathcal{K}; \mathbb{Z})\cong H^2(\mathcal{K};\mathbb{Z})=\Lambda$ with $\delta.\delta=-2$, we define $\xi^{\perp}$ to be subspace in $Gr^{+}$ consisting of hyperk\"ahler metrics $g$ such that $\int_{\delta}\xi=0$ for all $\xi\in \mathbb H^+_g$. We denote $$Gr^{+, \circ}=Gr^{+}\setminus \bigcup_{\delta\in \Lambda, \delta.\delta=-2} \delta^{\perp}$$
  and $\mathcal D^{\circ}=Gr^{+, \circ}/\Gamma$.

\begin{thm}\label{main theorem}
The image of $\underline{\mathcal P}$ is $\D^\circ$.

\end{thm}

\textbf{Acknowledgements.} The author thanks Song Sun for suggesting the problem and pointing out that the results in \cite{sun2021collapsing} may help in this problem.
The author is supported by Simons Collaboration Grant on Special Holonomy in Geometry, Analysis, and Physics (488633, S.S.).

\

We prove Theorem \ref{main theorem} in a few steps. 

\

{\textbf {Step 1}}. We show that the image of $\mathcal P$ is contained in $Gr^{+, \circ}$. In particular, the image of $\underline{\mathcal P}$ is contained in $\mathcal D^\circ$.

\

The proof that we know uses some complex geometry. Suppose
 there is a $g\in \mathcal N$ and $0\neq \delta\in H_2(\mathcal K;\mathbb Z)$ such that  $\delta.\delta\geq -2$ and $\int_{\delta}\xi=0$ for all $\xi\in \mathbb H_g^+$. Choose a  hyperk\"ahler triple $\bm\omega=(\omega_1, \omega_2, \omega_3)$ for $g$. Then with respect to a compatible complex structure $J$ we see $\omega_1$ is a K\"ahler form, and $\Omega=\omega_2+\sqrt{-1}\omega_3$ is a holomorphic volume form.  By the Hodge decomposition it follows that $\delta$ is a $(1,1)$ class with respect to $J$. So $\delta=c_1(L)$ for some non-trivial holomorphic line bundle $L$. By the Hirzebruch-Riemann-Roch theorem it follows that $h^0(\mathcal K, L)+h^0(\mathcal K, L^{-1})=h^1(\mathcal K, L)+2+\frac{1}{2}\delta^2>0$. Without loss of generality  we assume $L$ has a non-zero holomorphic section $S$. Its zero set is a complex curve dual to $\delta$. It follows that $\int_\delta\omega>0$. Contradiction.

\

\textbf{Step 2}: We show that ${\mathcal P}$ is an open map. In particular, the image of $\underline{\mathcal P}$ is open.

\

This follows from the standard deformation theory. We outline the argument, see \cite{sun2021collapsing} for details. Suppose $g\in \mathcal N$. We fix a hyperk\"ahler triple $\bm\omega$ associated to $g$. In the following, we will often identify an element in $\Omega_+^2\otimes\mathbb{R}^3$ (i.e., a triple of self-dual 2-forms) with a 3×3 matrix-valued function in $\Omega^0\otimes\mathbb{R}^{3\times 3}$: 
 $\bm\eta\in\Omega_+^2\otimes{\mathbb{R}^3}$ corresponds to $\bm A=(A_{\alpha\beta})\in\Omega^0\otimes\mathbb{R}^{3\times 3}$
if $\eta_\alpha=\sum_{\beta=1}^3 A_{\alpha\beta}\omega_\beta$, or concisely $\bm\eta= \bm A.\bm\omega$. We claim that for any fixed small triple of anti-self-dual harmonic 2-forms $\bm h^-\in \mathbb{H}_{g}^-\otimes\mathbb{R}^3$, $\bm\omega':=\bm\omega+{\bm h}^-+{\bm h}^++dd^*(\bm{f}.\bm\omega)$ defines a hyperk\"ahler triple for some small $(\bm h^+,\bm f)\in \mathbb{H}_{g}^+\otimes\mathbb{R}^3\oplus C^{2,\gamma}(\Omega_+^2\otimes\mathbb{R}^3)$.

 Denote $\mathscr{S}_0(\mathbb{R}^3)$ the set of trace-free $3\times 3$ symmetric matrices, and $\mathfrak{F}$ the inverse of the map $\mathscr{S}_0(\mathbb{R}^3)\rightarrow \mathscr{S}_0(\mathbb{R}^3)$, $A\mapsto \text{tf}(A+A^T+AA^T)$ near 0, where $\text{tf}(B)=B-\frac{1}{3}\text{Tr}(B)$. Denote $\mathfrak{A}\subset C^{2,\gamma}(\Omega_+^2\otimes\mathbb{R}^3)$, $\mathfrak{B}\subset C^{\gamma}(\Omega_+^2\otimes\mathbb{R}^3)$, $\mathfrak{C}\subset \mathbb{H}_{g}^+\otimes\mathbb{R}^3$ denote the subspace consisting of trace-free symmetric matrices, respectively. For $\bm u=(\bm h^+,\bm f)\in\mathfrak{C}\oplus\mathfrak{A}$, define $\mathscr{F}:\mathfrak{C}\oplus\mathfrak{A}\rightarrow \mathfrak{B}$ by $\mathscr{F}(\bm u):={\bm h}^++d^+d^*(\bm f.\bm\omega)-\mathfrak{F}(-\text{tf}(S_{{\bm h}^-+d^-d^*(\bm f.\bm \omega)}))$, where $S_{\bm \theta^{-}}=(\theta^-_\alpha\wedge\theta^-_\beta/2\text{dVol}_g)$. The condition $\bm\omega'$ being hyperk\"ahler is equivalent to the equation $\mathscr{F}(\bm u)=0$.

To solve the equation, we write $\mathscr{F}(\bm u)=\mathscr{L}(\bm u)+\mathscr{N}(\bm u)$,  where $\mathscr{L}(\bm u)={\bm h}^++d^+d^*(\bm f.\bm \omega)=\bm h^++(\Delta_{g}\bm f).\bm\omega$, $\mathscr{N}(\bm u)=-\mathfrak{F}(-\text{tf}(S_{{\bm h}^-+d^-d^*(\bm f.\bm \omega)}))$. Then by standard elliptic theory that $\mathscr{L}$ is a bounded linear map which is surjective with a bounded right inverse, and $\|\mathscr{N}(\bm u)-\mathscr{N}(\bm v)\|\leq C(\| \bm h^-\|+\|\bm u\|+\|\bm v\|)(\|\bm u-\bm v\|)$. Then the implicit function theorem implies that there exists a $\delta>0$ such that for any $\bm \|\bm h^-\|< \delta$, $\mathscr{F}(\bm u)=0$ has a solution $\bm u$ with $\bm \|\bm u\|< C(\delta)$, which finishes the proof of the claim. Now the map $\Psi: \mathbb{H}_{g }^-\otimes\mathbb{R}^3\rightarrow Gr^+$, $\bm h^-\mapsto \text{span}\{\bm \omega+\bm h^-\}$ defines a homeomorphism from a neighborhood of 0 to a neighborhood of $\mathcal{P}(g)$, and $\text{span}\{\bm \omega+\bm h^-\}=\text{span}\{\bm \omega+\bm h^-+\bm h^++dd^*(\bm f.\bm \omega)\}$ as elements in $Gr^+$, it follows that the image of a neighborhood of $g$ under $\mathcal{P}$ contains a neighborhood of $\mathcal{P}(g)$, hence $\mathcal{P}$ is an open map.

\

\textbf{Step 3.}
We show that $\underline {\mathcal P}: \M\rightarrow \D^\circ $ is a proper map. 

\

 Suppose otherwise, then we can find a sequence of hyperk\"ahler metrics $g_j\in \mathcal N$ which do not converge smoothly modulo $\textbf{Diff}(\mathcal K)$, but
 there exist $\gamma_j\in \Gamma$ such that $\gamma_j. \mathcal P(g_j)$ converges to a positive 3-dimensional subspace $P_\infty$ in $ Gr^{+,\circ}$. Choose a  hyperk\"ahler triple $\bm\omega_j=(\omega_{j,1}, \omega_{j,2}, \omega_{j,3})$ for $g_j$.  Denote  $v_j^4=2\text{Vol}(g_j)$.
 We define the renormalized triple $\widetilde{\bm\omega}_j=v_j^{-2}\bm\omega_j$, then $\int_{\mathcal K}\widetilde\omega_{j,\alpha}\wedge\widetilde\omega_{j,\beta}=\delta_{\alpha\beta}$. %So $\gamma_j.[\widetilde{ \omega}_{jk}]\cup \gamma_j.[\widetilde{ \omega}_{jl}]=\delta_{kl}$.
 
\

Now we fix an norm $\|\cdot\|$ on $\Lambda_\mathbb R$. By abusing notation we also denote by $\|\cdot\|$ the standard norm on $\mathbb R^3$, or the induced norm on $\Lambda_\mathbb{R}\otimes\mathbb{R}^3$.
 
\

 \begin{lem} $\|\gamma_j.[\widetilde{\bm\omega}_j]\|$ is uniformly bounded. 
 \end{lem}
 \begin{proof}Otherwise, by passing to a subsequence and  $O(3)$ rotations we may assume only the first component of $\gamma_j.[\widetilde{\bm\omega}_j]$ is non-zero and $\|\gamma_j.[\widetilde{\bm\omega}_{j}]\|=\lambda_j\rightarrow\infty$. Denote $\zeta_j=\lambda_j^{-1}\gamma_j.[\widetilde{\bm\omega}_{j}]$, then passing to a subsequence we may assume $\zeta_j$ converges to an element $\zeta_\infty$ in $\Lambda_{\mathbb R}$ with $\|\zeta_\infty\|=1$ and with $\zeta_\infty\cup \zeta_\infty=0$. But the line spanned by $\zeta_\infty$ is contained in $P_\infty$ which is positive definite with respect to the intersection form. Contradiction.
 \end{proof}
 
  Given the Lemma, by passing to a subsequence we may assume $\gamma_j.[\widetilde{\bm\omega}_j]$ converges to a limit $\bm{\eta}_\infty$ in $\Lambda_{\mathbb R}\otimes \mathbb R^3$. Notice $[\eta_{\infty,\alpha}]\cup [\eta_{\infty,\beta}]=\delta_{\alpha\beta}$, so $\bm\eta_\infty$ forms a basis for $P_\infty$.

\begin{prop}
	Passing to a subsequence, for $j$ large, there exists a non-zero homology class $C_j\in H_2(\mathcal K, \mathbb Z)$ satisfying $C_j.C_j\in \{0, -2\}$, and $\|\int_{C_j}\widetilde {\bm \omega}_j\|\rightarrow 0$.
\end{prop}

\begin{proof}
 Passing to a subsequence we may assume $(\mathcal K, g_j)$ converges to a Gromov-Hausdorff limit $X_\infty$, which is a compact metric space.

If $v_j\geq \epsilon>0$ for all $j$, then it follows from the classical results (\cite{anderson1989ricci, bando1989construction,tian1990calabi}) that  $X_\infty$ is a hyperk\"ahler orbifold. Let $p_j\in \mathcal K$ be such that $\lambda_j:=\max_{\mathcal K}|Rm(g_j)|$ is achieved at $p_j$. By assumption $\lambda_j\rightarrow\infty$. Then passing to a subsequence we can take a pointed Gromov-Hausdorff limit of $(\mathcal K, p_j, \lambda_j^{1/2}g_j)$ to get a complete  ALE hyperk\"ahler 4-manifold $Z$. By Kronheimer's classification \cite{kronheimer1989torelli} we know $Z$ must contain a homology class $C_\infty$ with $C_\infty.C_\infty=-2$. This gives rise to a sequence of $-2$ class $C_j$ in $\mathcal K$ such that $\|\int_{C_j}\bm\omega_j \|\rightarrow$ 0, so $\|\int_{C_j}\widetilde {\bm \omega}_j\|\rightarrow 0$ as well. Notice here we only need the topological classification in Kronheimer's result.

If $v_j\rightarrow 0$, then 
the conclusion follows from the results of Sun-Zhang \cite{sun2021collapsing}. The point is that away from finitely many points, the hyperk\"ahler triple $\bm\omega_j$(up to $O(3)$ rotations) has almost local nilpotent symmetry, and it can be perturbed  to a new hyperk\"ahler triple $\bm\omega_j'$ such that $\bm\omega_j'$ has local nilpotent symmetry and $\bm\omega_j-\bm\omega_j'$ is exact. In particular, $\bm\omega_j'$ has explicit expression and the integration of $\bm\omega_j'$ over a cycle is the same as the integration of $\bm\omega_j$ over that cycle. We divide into 3 cases. The first two cases only use the analysis over the regular region in \cite{sun2021collapsing}.

\begin{itemize}
    \item $\dim X_\infty=2$.  In this case $\bm\omega_j'$ is locally $T^2$-invariant. Take  $C_j$ to be the class of a $T^2$ fiber. Then we get $\int_{C_j}\omega_{j,1}=\int_{C_j}\omega_{j,2}=0$ and $\int_{C_j} \omega_{j,3}\sim v_j^4$. We get $\|\int_{C_j}\widetilde{\bm \omega}_j\|\sim v_j^2\rightarrow0.$

\item $\dim X_\infty=1$. Locally there are two cases, either there is a $T^3$ symmetry or a Heisenberg symmetry. In the former case  the metric corresponding to  $\bm\omega_j'$ is locally a flat product $S^1(r_{j,1})\times S^1(r_{j,2})\times S^1(r_{j,3})\times I$, where $I$ is an open interval and we assume $r_{j,1}\leq r_{j,2}\leq r_{j, 3}$, then we take $C_j$ to be the homology class  of $S^1(r_{j,1})\times S^1(r_{j,2})$, then $\int_{C_j}\omega_{j,1}=\int_{C_j}\omega_{j,2}=0$ and $0<\int_{C_j}\omega_{j,3}=r_{j, 1}r_{j,2}$ whereas  $v_j^4\sim r_{j,1}r_{j,2}r_{j,3}$. So $\|\int_{C_j}\widetilde{\bm \omega}_j\|\sim r_{j,1}^{1/2}r_{j_2}^{1/2}r_{j,3}^{-1/2}\rightarrow0$.	In the latter case the metric corresponding to $\bm\omega_j'$ is locally given by the Gibbons-Hawking ansatz applied to a nonconstant linear function on $T^2\times I$. Take $C_j$ to be the homology class of the 2-torus given by the total space of the corresponding circle bundle over a circle $S^1$ in $T^2\times I$, then one can arrange that $\int_{C_j}\omega_{j,1}=\int_{C_j}\omega_{j,2}=0$ and $0<\int_{C_j}\omega_{j,3}\sim r_{j,1}r_{j,2}$, where $r_{j,1}$ is the size of the $S^1$ fiber in the Gibbons-Hawking construction, and $r_{j,2}$ is the size of the flat $T^2$ base. Notice the volume  $v_j^4\sim r_{j,1}r_{j,2}^2$. So $\|\int_{C_j}\widetilde{\bm \omega}_j\|\sim r_{j,1}^{1/2}\rightarrow0$.
\item $\dim X_\infty=3$. Here we need some global result from \cite{sun2021collapsing}. It is proved there that $X_\infty$ must be a flat orbifold $T^3/\mathbb Z_2$. Then  $\bm\omega_j'$ is given by Gibbons-Hawking construction on the complement of a small neighborhood of the orbifold points. We can take $C_j$ to be the total space of the circle bundle over a closed geodesic $S^1\subset T^3/\mathbb Z_2$. Then we may arrange that $\int_{C_j}\omega_{j,1}=\int_{C_j}\omega_{j,2}=0$ while $0<\int_{C_j}\omega_{j,3}=r_{j}\sim v_j^4$. So $\|\int_{C_j}\widetilde{\bm \omega}_j\|\sim r_j^{1/2}\rightarrow0.$ 
  \end{itemize}

\end{proof}

 Now we derive a contradiction. Denote $C_j'=\gamma_j^{-1}.C_j$. Since $C_j'$ is integral and non-zero, we know $\|C_j'\|$ has a uniform positive lower bound.  By passing to a subsequence we may assume $\|C_j'\|^{-1}C_j'$ converges to a limit $C_\infty'\in \Lambda_{\mathbb R}$ with $\|C_\infty'\|=1$. First suppose $C_j.C_j=0$, then   $C_\infty'. C_\infty'=0$ and 
 \begin{equation}\label{equality}
\int_{C_\infty'}\bm\eta_\infty=\lim\limits_{j\rightarrow\infty}\|C_j'\|^{-1}\int_{C_j'} \bm\eta_{\infty}=\lim\limits_{j\rightarrow\infty}\|C_j'\|^{-1}\int_{C_j'} \gamma_j.[\widetilde{\bm\omega}_j]=\lim\limits_{j\rightarrow\infty}\|C_j'\|^{-1}\int_{C_j}\widetilde{\bm\omega}_j=0
 \end{equation} 
This contradicts the fact that the intersection form on $\Lambda_{\mathbb R}$ has signature $(3, 19)$. 
Now suppose $C_j. C_j=-2$. If $\|C_j'\|$ is unbounded,
then $C_\infty'.C_\infty'=0$, and we have $\int_{C_\infty'}\bm\eta_\infty=0$ as in (\ref{equality}), hence we obtain a contradiction in the same way. If $\|C_j'\|$ is bounded, by passing to a further subsequence we may assume $C_j'$ converges to a limit $C_\infty''\in \Lambda $ with $C_\infty''. C_\infty''=-2$ and  $\int_{C_\infty''}\bm\eta_\infty=0$. This contradicts to $P_\infty \in Gr^{+,\circ}$ and finishes the proof of Step 3.

Finally, $\underline{\mathcal{P}}:\M\rightarrow \D^\circ $ being proper and $\D^\circ$ being locally compact, Hausdorff imply $\underline{\mathcal{P}}:\M\rightarrow \D^\circ $ is a closed map. Together with the image of $\underline{\mathcal{P}}$ being open, $\mathcal{D}^\circ$ being connected, we conclude $ \underline{\mathcal{P}}:\M\rightarrow \D^\circ $ is surjective.

\end{document}